\documentclass{article}
\usepackage{latexsym}
\usepackage{amsmath,amsthm}
\usepackage{amssymb}
\usepackage{graphics}
\usepackage{color}

\newtheorem{theorem}{Theorem}[section]
\newtheorem{proposition}{Proposition}[section]

\newtheorem{corollary}{Corollary}[section]

\def\pa{\partial}

\newcommand{\nabb}{\mbox{$\nabla \mkern-13mu /$\,}}

\begin{document}
\title{A note on energy currents and decay for the wave equation 
on a Schwarzschild background}
\author{Mihalis Dafermos\thanks{University of Cambridge,
Department of Pure Mathematics and Mathematical Statistics,
Wilberforce Road, Cambridge CB3 0WB United Kingdom}
\and
 Igor Rodnianski\thanks{Princeton University,
Department of Mathematics, Fine Hall, Washington Road,
Princeton, NJ 08544 United States}
}
\maketitle
\begin{abstract}In recent work, we have proven uniform decay bounds
for
solutions of the wave equation
$\Box_g\phi=0$ on a Schwarzschild exterior, in particular,
the uniform pointwise estimate $|\phi|\le Cv_+^{-1}$ which holds throughout the domain of
outer communications, where
$v$ is an advanced Eddington-Finkelstein coordinate, $v_+\doteq\max\{v,1\}$,
and $C$ is a constant depending on a Sobolev norm of initial data.
A crucial estimate in the proof 
required a decomposition into spherical harmonics. We here give an alternative
proof of this estimate not requiring such a decomposition. 
\end{abstract}

In~\cite{dr3}, 
we studied the problem of decay for general solutions $\phi$ of the equation
\begin{equation}
\label{waveeq}
\Box_g\phi=0
\end{equation}
on a Schwarzschild background.
 The estimates of~\cite{dr3} were obtained by exploiting compatible currents
associated with vector field mutipliers applied to the energy momentum tensor. (See~\cite{book2}
for a general discussion of such currents.)
Understanding the decay properties of solutions of $(\ref{waveeq})$ in terms of such energy estimates 
appears to be a fundamental first step, if one 
is ever to address the problem of non-linear stability of black hole solutions of the
Einstein equations of general relativity.

A crucial role in the results of~\cite{dr3} is played by an energy current $J_\mu$ related
to vector fields of the form $f(r^*)\partial_{r^*}$, where $r^*$ is a Regge-Wheeler 
coordinate. For the current $J_\mu$ constructed in~\cite{dr3}, the
divergence $K=\nabla^\mu J_\mu$
was shown to be nonnegative upon integration over spheres of symmetry.
(The integral of this current over an arbitrary spacetime region $\mathcal{R}$
was denoted in~\cite{dr3} by $I^X(\mathcal{R})$.)
The construction of $J_\mu$ was quite elaborate. In particular,
a decomposition into spherical harmonics was required, 
and a separate definition of $f_\ell$ was made for each spherical harmonic, characterized 
by a non-negative integer $\ell$.
These currents were then summed to obtain a total current.
As $\ell\to \infty$, the unique vanishing point of $f_\ell$ approached the \emph{photon
sphere} $r=3M$. In this sense, the degeneration of the current was seen to be connected
to the presence of trapped null geodesics. 
This degeneration is known to be an essential feature
in view of well-known arguments from geometric optics.

With an eye towards possible future applications to 
problems involving perturbed spacetimes, it is desirable for methods which avoid 
altogether the use of spherical harmonics. In this short paper,
we indeed
construct a current $J_\mu$ with the required non-negativity properties for its
divergence $K=\nabla^\mu J_\mu$, without recourse to spherical harmonics.
See Propositions~\ref{nnprop},~\ref{qcprop} and~\ref{qcprop2}.
The current and its divergence both depend on the $2$-jet of the solution $\phi$.
This current can be substituted in the arguments of~\cite{dr3}, completely
removing references to spherical harmonics from the proof.
See Theorem~\ref{estthe} of Section~\ref{estsec} for the basic estimate.

Numerical evidence for the existence of currents with the nonnegativity property
achieved here has been presented recently in~\cite{bs}.

One should note finally that an easy perturbation argument shows that the
current $J_\mu$ defined above
can also be applied to 
the study of $(\ref{waveeq})$ on Schwarzschild-de Sitter, at least for small $M\sqrt{\Lambda}$,
thus removing references to spherical harmonics from the proof of Theorem
1.1 of~\cite{dr4}
for this case. See Section~\ref{comlab}.

\section{Energy current templates from vector fields}
Our notation follows closely that of our~\cite{dr4}. Recall that in the domain of outer
communications $\mathcal{D}$ the Schwarzschild metric can be written explicitly in a coordinate system
$(t,r)$:
$$
g=-\left(1-\frac {2M}r\right) dt^2 + \left(1-\frac {2M}r\right)^{-1} dr^2 + r^2 d\sigma_{{\Bbb S}^2},
$$
with coordinate range $(\infty,\infty)\times(2M,\infty)$. 
In what follows we set
$$
\mu = \frac {2M}r, \qquad r^*=r+2M \ln (r-2M)-3M-2M\ln M,
$$
and we let $'$ denote the derivative with respect to $r^*$.
Recall that (in our conventions)
Regge-Wheeler coordinates $(t,r^*)$ are related to
Eddington-Finkelstein coordinates $(u,v)$ 
by the formulas
\[
t=v+u,\qquad r^*=v-u.
\]

In this paper, 
$\phi$ will always denote a solution to the wave equation $(\ref{waveeq})$
on maximally extended Schwarzschild  $(\mathcal{M},g)$ which is $H^2$ on spheres
of symmetry and such that $\nabla\phi$ is in $L^2$ on spheres of symmetry.\footnote{We
require this regularity for we shall give integral bounds on spheres of symmetry. By integrating
these bounds in spacetime regions, one can apply these results to solutions of $(\ref{waveeq})$
with locally $H^2$ initial data.}
Given  $\phi$, let
$T_{\mu\nu}(\phi)$ denote the energy momentum tensor
\[
T_{\mu\nu}(\phi)=\partial_\mu\phi\partial_\nu\phi -\frac12 g_{\mu\nu}g^{\alpha\beta}\partial_\alpha
\phi\partial_\beta\phi.
\]
Let $f$ be a function of $r^*$ and consider a vector field of the form
\begin{equation}\label{eq:V}
V= f (r^*)\frac{\partial}{\partial r^*},
\end{equation}
for an arbitrary function $f$.

Let the function $\beta$ be defined by
$$
\beta=\frac {1-\mu}r -\frac {x}{\alpha^2+x^2},
$$
where 
$$
x=r^*-\alpha-\alpha^{\frac 12}
$$
for  a (sufficiently large) constant $\alpha$ to be determined later.
Define the currents
\begin{eqnarray*}
J_\mu ^{V,0}(\phi)&=& T_{\mu\nu}(\phi){V}^\nu,\\
J_\mu^{V, 1}(\phi) &=& T_{\mu\nu}(\phi){V}^\nu+\frac 14\left(f'+2\frac{1-\mu}r f\right) \pa_\mu (\phi)^2-
\frac 14 \pa_\mu \left(f'+2\frac{1-\mu}r f\right) \phi^2,\\
J_\mu^{V,2}(\phi) &=&T_{\mu\nu}(\phi)V^\nu +\frac 14\left(f'+2\frac{1-\mu}r f\right) \pa_\mu (\phi)^2-
\frac 14 \pa_\mu \left(f'+2\frac{1-\mu}r f\right) \phi^2\\
&&\hbox{}-\frac 12 \frac {f'}{ f(1-\mu)}\beta V_\mu \phi^2,
\end{eqnarray*}
and the divergences
\[
K^{V,i}= \nabla^\mu J_{\mu}^{V,i}.
\]

We compute
\begin{align}
K^{V,0}(\phi) &=
\frac{f'(\partial_{r^*}\phi)^2}{1-\mu}
+|\nabb\phi|^2\left(\frac{\mu'}{2(1-\mu)}+
		\frac{1-\mu}{r}\right)f\notag\\
&-\frac14\left(2f'+4\frac{1-\mu}r f \right)\phi^\alpha
\phi_\alpha,\label{eq:KV0}
\end{align}
\begin{eqnarray*}
K^{V,1}(\phi) &=&
	\frac{f'}{1-\mu}(\partial_{r^*}\phi)^2
		+|\nabb\phi|^2\left(\frac{\mu'}{2(1-\mu)}+
		\frac{1-\mu}{r}\right)f\\
		&&\hbox{}-\frac 14 \left (\Box \left(f'+2\frac{1-\mu}rf\right) 
		\right) \phi^2\\
		&=& 	\frac{f'}{1-\mu}(\partial_{r^*}\phi)^2
		+|\nabb\phi|^2\left(\frac{\mu'}{2(1-\mu)}+
		\frac{1-\mu}{r}\right)f\\
		&&\hbox{}-\frac 14 \left (
		 \frac{1}{1-\mu}f'''+\frac 4rf''-\frac{4\mu'}{r(1-\mu)}f'
	+\frac{2}{(1-\mu)r}\left(\frac{\mu'(1-\mu)}r-\mu''\right)f
		\right)\phi^2,
\end{eqnarray*}
\begin{eqnarray*}
K^{V,2}(\phi) &=&
\frac{f'}{(1-\mu)}\left(\partial_{r^*} \phi+\beta\phi\right)^2+
\frac{r-3M}{r^2} \, f\, |\nabb\phi|^2\\ &&-\frac14\frac1{1-\mu}\left({f'''}+\frac{4f''x}{\alpha^2+x^2}+
\frac{4\alpha^2 f'}{(\alpha^2+x^2)^2}\right)\phi^2
-\frac {\mu f}{2r^3}(4\mu-3)\phi^2. 
\end{eqnarray*}

\section{Definition of the current $J_\mu$}
Let $\Omega_i$, $i=1,\ldots, 3$ denote a basis of angular momentum operators.

Define
\[
f^a= -\frac {C_*}{\alpha^2 r^2} 
\]
for a constant $C_*$, dependent on $\alpha$, both of which to be determined in what follows,
\[
f^b= \frac {1}{\alpha} \left (\tan^{-1}\frac x\alpha-\tan^{-1}(-1-\alpha^{-\frac 12})\right),\qquad
(f^b)'=\frac 1{\alpha^2+x^2}
\]
\[
X^a= f^a\partial_{r^*}
\]
\[
X^b= f^b\partial_{r^*}
\]
and let
\[
J = J^{X^a, 0}(\phi)+\sum_{i=1}^3J^{X^b,1}(\Omega_i\phi),
\]
\[
K= \nabla^\mu J_\mu.
\]
Note that $J$ and $K$ both depend on the $2$-jet of $\phi$.

\section{Nonnegativity}
\begin{proposition}
\label{nnprop}
For the $K$ defined above with $\alpha$, $C_*$ suitably chosen,
\[
\int_{\mathbb S^2} K\, r^2 dA_{\mathbb S^2}\ge 0 .
\]
\end{proposition}
\begin{proof}

Note that $(f^a)'\ge 0$, $(f^b)'\ge 0$ and $2(f^a)'+4\frac{1-\mu}rf^a=0$.
We then conclude that 
\begin{equation}
\label{inviewof}
K\ge \sum_{i=1}^3 \left (f^b \frac {r-3M}{r^2} |\nabb\Omega_i\phi|^2 + 
\left (f^b\frac {\mu(3-4\mu)}{2r^3} +F\right) (\Omega_i\phi)^2\right) -
{2C_*}\frac{r-3M}{\alpha^2 r^4} |\nabb\phi|^2 
\end{equation}
where
$$
F:=-\frac14\frac1{1-\mu}\left({(f^b)'''}+\frac{4(f^b)''x}{\alpha^2+x^2}+
\frac{4\alpha^2 (f^b)'}{(\alpha^2+x^2)^2}\right)=
\frac 1{2(1-\mu)}\frac {x^2-\alpha^2}{(x^2+\alpha^2)^3}.
$$

Note that
$
\sum_{i=1}^3 
 (\Omega_i\phi)^2 
=
 r^2|\nabb\phi|^2   $.
Note also the Poincar\'e inequality for $\Omega_i\phi$
\begin{equation}
\label{poincare}
2\int_{\mathbb S^2} (\Omega_i\phi)^2 \,r^2 dA_{\mathbb S^2}  \le 
 r^2 \int_{\mathbb S^2} |\nabb\Omega_i\phi|^2  r^2 dA_{\mathbb S^2}.
\end{equation}

Thus, to prove the proposition, in view of $(\ref{inviewof})$,
it suffices to show
\begin{equation}
\label{suffices}
2f^b\frac{r-3M}{r^4}+
\left (f^b\frac {\mu(3-4\mu)}{2r^3} +F\right) - 
{2C_*}\frac{r-3M}{\alpha^2 r^6} \ge 0.
\end{equation} 
or alternatively,
\begin{equation}
\label{suffices2}
2 f^b(r-3M)r^2+
H(r) + Fr^6 \ge 0
\end{equation}
where
$$
H(r):=f^b\frac {\mu(3-4\mu)}{2}r^3-{2C_*}\frac{r-3M}{\alpha^2 }=
r^4\left (f^b\frac {\mu(3-4\mu)}{2 r^3}r^2-{2C_*}\frac{r-3M}{\alpha^2 r^4}\right).
$$

Note that the first term on the right hand side of
$(\ref{suffices2})$ is manifestly nonnegative.

We first establish that the function $H$ satisfies $H(r)\ge 0$.
Observe that $H> 0$ in the region $r\le 8M/3$ and that $H(3M)=0$. 
With the constant $C_*$ chosen such that $\frac{dH}{dr}|_{r=3M}=0$,
in order to show that $H\ge 0$ in the range $8M/3\le r<\infty$,
it suffices to show that there exists an $R$ such that $H\ge 0$
for $r\ge R$, and $\frac{d^2H}{dr^2}\ge 0$ in
$8M/3\le r\le R$.

We compute
$$
\frac {dH}{dr}=M (f^b)'\frac {3r^2-8Mr}{1-\mu} + 2M f^b(3r-4M)-2C_* \alpha^{-2}.  
$$
Setting $\frac{dH}{dr}|_{r=3M}=0$ implies 
$$
C_*=\alpha^2 M  (f^b)'\frac {3r^2-8Mr}{2(1-\mu)}|_{r=3M}=
 \frac{9\alpha^2 M^3}{2((\alpha+\alpha^{\frac 12})^2+\alpha^2)}.
$$
Given $\alpha$, let this then be our choice of $C_*$.
It is clear that as $\alpha\to \infty$, 
\[
C_*\to\frac {9M^3}{4}.
\]
Note that from the definition of $f^b$ we have that
there exist constants $c$, $R$, independent of $\alpha$, such that
for all values $r\ge R$
\begin{equation}
\label{fbbound}
f^b\ge  \frac c\alpha \min \left\{\frac r\alpha,1\right\}. 
\end{equation}
As a consequence, taking into account the value of the
constant $C_*$, we see that
for $r\ge R$,
$$
H(r) \ge \frac {c M}{\alpha} (3r^2-8Mr)   \min \left\{\frac r\alpha,1\right\} -
 {2C_*}\frac{r-3M}{\alpha^2 }>0,
$$

To show $H\ge 0$, 
it remains to check that $\frac {d^2H}{dr^2}\ge 0$ for all 
$r\in [\frac {8M}3,R]$.

We compute
$$
\frac {d^2H}{dr^2}=M (f^b)''\frac {3r^2-8Mr}{(1-\mu)^2} + 4M (f^b)'\frac{3r-4M}{1-\mu}
-2M^2 (f^b)'\frac{3r^2-8Mr}{r^2(1-\mu)^2}+6Mf^b.
$$
First, we easily see that for $r\in [\frac {8M}3,R]$,
$$
M (f^b)''\frac {3r^2-8Mr}{(1-\mu)^2}\le C \alpha^{-3},
$$
for a $C$ independent of $\alpha$.
As a consequence, this term will be dominated (for sufficiently
large $\alpha$ by the other terms in the expression
for $\frac {d^2H}{dr^2}$, which are of the order of $\alpha^{-2}$.

We combine the terms containing $(f^b)'$, 
taking into account that for $r\in [\frac {8M}3,R]$ we have $(1-\mu)\ge \frac 14$, 
to obtain
\begin{align*}
4M (f^b)'\frac{3r-4M}{1-\mu}
-2M^2 (f^b)'\frac{3r-8M}{r(1-\mu)^2}&= 
\frac{2M (f^b)'}{r(1-\mu)}\left (6r^2-4Mr - M\frac{3r-8M}{1-\mu}\right)\\ &\ge 
\frac{2M (f^b)'}{r(1-\mu)}
\left (6r^2-16Mr + 32M^2\right)\ge c\alpha^{-2}.
\end{align*}
Moreover, in the region $r\in [3M,R]$ 
the last term $6Mf^b$ is non-negative, which immediately implies desired conclusion that 
$\frac {d^2H}{dr^2}>0$. On the other hand, for $r\in [\frac {8M}3,3M]$ we have 
$$
f^b=\int_{r}^{3M} \frac{(f^b)'}{1-\mu} dr \ge \frac {(f^b)'|_{r=3M}}{1-\mu} (r-3M). 
$$ 
From the expression for $(f^b)'=(\alpha^2+x^2)^{-1}$ with 
$x=r^*-\alpha-\alpha^{\frac12}$ we easily see
that for all $r\in [\frac {8M}3,3M]$
$$
(f^b)'|_{r=3M}=(f^b)'(r) + O(\alpha^{-3}).
$$
Therefore,
$$
\frac {d^2H}{dr^2}\ge \frac{2M (f^b)'}{r(1-\mu)}\left (9r^2-25Mr + 32M^2\right) + O(\alpha^{-3})\ge 
c\alpha^{-2}
$$
for $r\in [\frac {8M}3,3M]$.

To prove $(\ref{suffices2})$, and thus Proposition~\ref{nnprop} 
it now suffices to establish the inequality
$$
2f^b\frac{r-3M}{r^4}+F\ge 0. 
$$

We note that the function
$F$ is non-negative outside the region $-\alpha<x<\alpha$.
Thus, in view of the nonnegativity of the first term
above, it follows that $(\ref{suffices2})$ holds
for $x\not\in[-\alpha,\alpha]$.

For $x\in[-\alpha,\alpha]$ on the other hand,  we have that 
$r^*\in [\alpha^{\frac 12},2\alpha+\alpha^{\frac 12}]$, which implies that, for $\alpha$ sufficiently
large, 
$$
r\ge c \alpha^{\frac 12},\qquad 
r=\left (1+O(\alpha^{-\frac 12})\right) (x+\alpha+\alpha^{\frac12}),\qquad \mu=O(\alpha^{-\frac 12}).
$$
We may approximate functions $F$  and $2f^b\frac{r-3M}{r^4}$ by the expressions
$$
F\sim \frac 12 \frac {x^2-\alpha^2}{(x^2+\alpha^2)^3},\qquad 
2f^b\frac{r-3M}{r^4}\sim  \frac {2(x+\alpha)}{(x^2+\alpha^2)
(x+\alpha+\alpha^{\frac 12})^3}.
$$
It suffices then to establish the bound 
$$
\frac {(\alpha-x)(x+\alpha+\alpha^{\frac12})^3}{4(x^2+\alpha^2)^2}<\frac 9{10}, \qquad \forall
x\in [-\alpha,\alpha].
$$

For $-\alpha\le x\le 0$ we have 
$$(\alpha-x)\le 2\alpha,\quad 
(x+\alpha+\alpha^{\frac 12})^3<\frac 32\alpha^3,\quad
(x^2+\alpha^2)\ge \alpha^2,
$$
where the middle inequality follows if the constant $\alpha$ is chosen to be 
sufficiently large.
Therefore,
$$
\frac {(\alpha-x)(x+\alpha+\alpha^{\frac12})^3}{4(x^2+\alpha^2)^2}<\frac 3{4}.
$$
On the other hand, for $0\le x\le \alpha$ we have 
$$
(x+\alpha+\alpha^{\frac 12})^3 < 2^{\frac 32}\, \frac {8}{7} (x^2 +\alpha^2)^{\frac 32},
\qquad \alpha-x\le \alpha.
$$
Thus,
$$
\frac {(\alpha-x)(x+\alpha+\alpha^{\frac12})^3}{4(x^2+\alpha^2)^2}<\frac 2{7}2^{\frac 32}<\frac{9}{10}.
$$
\end{proof}

\section{Quantities controlled}
\begin{proposition}
\label{qcprop}
There exists a constant $C$ depending only on $M$ such that
\[
C\int_{\mathbb S^2} K\, r^2 dA_{\mathbb S^2}
\ge 
\int_{\mathbb S^2} \left(\frac1{r^3}\left (\pa_{r^*} \phi\right)^2+
\frac{(r-3M)^2}{r} |\nabb^2\phi|^2+ \frac {r^3}{(1-\mu)(|r^*|+1)^4} 
|\nabb\phi|^2\right)r^2 dA_{\mathbb S^2}.
\]
\end{proposition}
\begin{proof}
Revisit the proof of Proposition~\ref{nnprop} and recall the nonnegative quantities that were
dropped.
\end{proof}

Now define $X^{\rm aux}= r^{-3}\partial_{r^*}$ 
and define $J^{\rm aux}_\mu = J^{X^{\rm aux},0}_\mu$,
$K^{\rm aux} = \nabla^\mu J^{\rm aux}_\mu$. 
We easily see (cf.~Section~7.4 of~\cite{dr4})
\begin{proposition}
\label{qcprop2}
There exists a constant $C$ depending only on $M$ such that
\begin{align*}
C\int_{\mathbb S^2} (K^{\rm aux} +K) r^2dA_{\mathbb S}^2\\
\ge \int_{\mathbb S^2} \left (\frac 1{r^3} \left (\pa_{r^*} \phi\right)^2+
\frac{(r-3M)^2}{r} |\nabb^2\phi|^2+ \frac {r^3}{(1-\mu)(|r^*|+1)^4} 
|\nabb\phi|^2+\frac1{r^4}(\partial_t\phi)^2\right)r^2 dA_{\mathbb S^2}.
\end{align*}
\end{proposition}

\section{Boundary terms}
To turn Propositions 4.1, 4.2 into estimates with the 
help of the divergence theorem, we need to understand the boundary
terms arising from the integration of the currents $K$, $K^{\rm aux}$.
\begin{proposition}
\label{btp}
Let $\mathcal{S}$ be an achronal subset of ${\rm clos}(\mathcal{D})$
such that ${\rm clos}(\mathcal{S})\cap \mathcal{H}^-=\emptyset$.
There exist constants $\epsilon, C>0$,
depending only on  
$\mathcal{S}$ such that
\[
\left|\int_{\mathcal{S}} J^{\rm aux}_\mu n^\mu \right|+
\left|\int_{\mathcal{S}} J_\mu n^\mu \right|\le  \int_{\mathcal{S}}
\left(CJ^T_\mu(\phi)+\epsilon J^Y_\mu(\phi)+C\sum_i J^T_\mu(\Omega_i\phi)\right)n^\mu
\]
where $J^T_\mu=T^\nu T_{\mu\nu}$, $J^Y_\mu= Y^\nu T_{\mu\nu}$
where $T$ denotes the Killing field
$\frac\partial{\partial t}$ and $Y$ denotes the local observer vector
field of~\cite{dr3}.
In the case where $\mathcal{S}$ is spacelike, $n^\mu$ denotes the future
pointing unit normal and the measure is the volume form. In the case where
$\mathcal{S}$ is null, $n^\mu$ and the measure are defined appropriately.

If $\phi_t$ denotes the one-parameter group of transformations generated
by $T$, then one can take $C(\mathcal{S},\epsilon)=C(\phi_t(\mathcal{S}),\epsilon)$,
and $\epsilon(\mathcal{S},C)=\epsilon(\phi_t(\mathcal{S}),C)$.
If $\mathcal{S}$ is a constant $t$-surface or if 
${\rm clos}(\mathcal{S})\cap(\mathcal{H}^+\cup\mathcal{H}^-)=\emptyset$, then 
the $J^Y$ term can be omitted from the right hand side, i.e.~one can take
$\epsilon=0$.
If $\mathcal{S}$ is a constant $u$-surface such that $r^*\le r_0^*$, $t\ge t_0$
in $\mathcal{S}$,
then as $r_0^*\to-\infty$ with $t_0$ fixed
one can take $C$ uniformly bounded and $\epsilon\to0$.
\end{proposition}

\section{The estimates}
\label{estsec}
Consider a ``trapezoidal'' region $\mathcal{R}$
defined by the inequalities $t_1\le t\le t_2$, $r_1^*-(t_2-t)\le r^* \le
r_2^*+(t_2-t)$. Let $\mathcal{S}_1$ denote the timelike past boundary
$\{t_1\}\times[r_1^*-(t_2-t_1),r_2^*+(t_2-t_1)]$, and let $\mathcal{F}_1$
denote the constant-$t-r^*$ boundary.
Propositions~\ref{btp},~\ref{qcprop},~\ref{qcprop2}, the divergence
theorem and the fact
that $K^T=0$ give
\begin{theorem}
There exists a $C$ depending only on $M$ and an $\epsilon$ depending only
on $r_1^*$ with $\epsilon\to0$ as $r_1^*\to-\infty$ such that
\label{estthe}
\begin{align*}
\int_{\mathcal{R}} \left (\frac 1{r^3} \left (\pa_{r^*} \phi\right)^2+
\frac{(r-3M)^2}{r} |\nabb^2\phi|^2+ \frac {r^3}{(1-\mu)(|r^*|+1)^4} 
|\nabb\phi|^2+\frac{1}{r^4}(\partial_t\phi)^2\right)\\
\le C\int_{\mathcal{S}_1} \left(J^{T}_\mu(\phi)+\sum_i J_\mu^T(
\Omega_i\phi)\right) n^\mu
+\epsilon \int_{\mathcal{F}_1} J^Y_\mu(\phi)n^\mu.
\end{align*}
\end{theorem}

Let $\mathcal{D}$ denote the domain of outer communications.
Note that taking $r_1^*\to-\infty$, $t_2\to\infty$, an immediate
corollary of the above is
\begin{corollary}
There exists a constant $C$ depending only on $M$ such that
\begin{align*}
\int_{\mathcal{D}} 
\left (\frac 1{r^3} \left (\pa_{r^*} \phi\right)^2+
\frac{(r-3M)^2}{r} |\nabb^2\phi|^2+ \frac {r^3}{(1-\mu)(|r^*|+1)^4} 
|\nabb\phi|^2+\frac{1}{r^4}(\partial_t\phi)^2\right)\\
\le C\int_{t=0} \left(J^{T}_\mu(\phi)+\sum_i J_\mu^T(
\Omega_i\phi)\right) n^\mu.
\end{align*}
\end{corollary}

On the other hand, using the $Y$ estimate of~\cite{dr3} one obtains
\begin{theorem}
There exists a constant $C$ depending only on $M$ such that
\begin{align*}
\int_{\mathcal{R}} \left (\frac 1{r^3} \left (\pa_{r^*} \phi\right)^2+
\frac{(r-3M)^2}{r} |\nabb^2\phi|^2+ \frac {r^3}{(1-\mu)(|r^*|+1)^4} 
|\nabb\phi|^2+\frac{1}{r^4}(\partial_t\phi)^2\right)\\
\le C\int_{\mathcal{S}_1} \left(J^{T+Y}_\mu(\phi)+\sum_i J_\mu^T(
\Omega_i\phi)\right) n^\mu.
\end{align*}
\end{theorem}
The above Theorem can be used in conjunction with the $Y$
estimates and the Morawetz vector field $u^2\partial_u +v^2\partial_v$
to obtain the results of~\cite{dr3}.
Note that, in contrast to the scheme of~\cite{dr3},
the vector field $u^2\partial_u+v^2\partial_v$ 
is not necessary to control the boundary
terms arising from $J^\mu$.

\section{Comments}
\label{comlab}
As noted in the introduction, by a simple perturbation argument,
our results also apply to remove reference to spherical harmonics from
the arguments of~\cite{dr4} in the case of small $M\sqrt{\Lambda}$.
This relies on the fact that error terms in the immediate vicinity
of the horizons can be absorbed via the use of the $Y$ and $\overline{Y}$ estimates.

It is also interesting to note that the currents based on the vector fields $f_\ell \partial_{r^*}$ 
of~\cite{dr3} were also
essential even to obtain just the uniform 
boundedness statement  $|\phi|\le C$, where $C$ is a constant
depending only on a norm of initial data.
Recall that this statement was originally proven 
by Kay and Wald~\cite{kw:lss},\footnote{In spite of the classical paper~\cite{kw:lss}, many
later authors continue to
use the term ``linear stability'' or ``decay'' to refer to results which are completely
compatible with the statement $\sup_{\mathcal{D}}|\phi|=\infty$ where $\mathcal{D}$ denotes
the domain of outer communications. It is hard even to imagine what kind of stability such results
would indicate.} with methods that relied
heavily on the staticity of exterior Schwarzschild and a certain discrete isometry
of the maximally extended solution.
In view of the results of the present paper, 
uniform boundedness can now be shown without the Kay-Wald trick,
without the vector field $u^2\partial_u+v^2\partial_v$ and without
recourse to spherical harmonics.

\subsection{Acknowledgements}
M.D.~thanks Princeton University for hospitality
in February 2007 when this research was conducted. M.D.~is supported
by a Clay Research Scholarship. I.R.~is supported in part by NSF grant
DMS-0702270.

\end{document}